\documentclass[12pt,a4paper]{amsart}

\usepackage{amssymb}
\usepackage{mathabx}
\usepackage{bbm}
\usepackage{amsthm}
\usepackage{dsfont}
\usepackage{amsmath}
\usepackage{color,graphics,srcltx}
\usepackage{bm}
\usepackage{easybmat}
\usepackage{multirow,bigdelim}
\usepackage{enumerate}
\usepackage{graphicx}

\newtheorem{proposition}{Proposition}

\newtheorem{lemma}[proposition]{Lemma}
\newtheorem{corollary}[proposition]{Corollary}
\newtheorem{theorem}[proposition]{Theorem}

\newcommand{\RR}{\mathbb{R}}

\def\cQ{\mathcal{Q}}

\definecolor{damjana}{rgb}{.8,.2,.2}

\def\RR{\mathds R}

\def\II{\mathds I}

\def\BB{\mathcal B}
\def\CC{\mathcal C}

\def\FF{\mathcal F}

\begin{document}

\title[Blomqvist's beta vs.\ measures of concordance]{Relation between Blomqvist's beta and other measures of concordance of copulas}

\author[D. Kokol B.]{Damjana Kokol Bukov\v{s}ek}
\address{School of Economics and Business, University of Ljubljana, and Institute of Mathematics, Physics and Mechanics, Ljubljana, Slovenia}
\email{damjana.kokol.bukovsek@ef.uni-lj.si}

\author[T. Ko\v sir]{Toma\v{z} Ko\v{s}ir}
\address{Faculty of Mathematics and Physics, University of Ljubljana, and Institute of Mathematics, Physics and Mechanics, Ljubljana, Slovenia}
\email{tomaz.kosir@fmf.uni-lj.si}

\author[B. Moj\v skerc]{Bla\v{z} Moj\v{s}kerc}
\address{School of Economics and Business, University of Ljubljana, and Institute of Mathematics, Physics and Mechanics, Ljubljana, Slovenia}
\email{blaz.mojskerc@ef.uni-lj.si}

\author[M. Omladi\v c]{Matja\v{z} Omladi\v{c}}
\address{Institute of Mathematics, Physics and Mechanics, Ljubljana, Slovenia}
\email{matjaz@omladic.net}

\begin{abstract}
An investigation is presented of how a comprehensive choice of four most important measures of concordance (namely Spearman's rho, Kendall's tau, Spearman's footrule, and Gini's gamma) relate to the fifth one, i.e., the Blomqvist's beta. In order to work out these results we present a novel method of estimating the values of the four measures of concordance on a family of copulas with fixed value of beta. These results are primarily aimed at the community of practitioners trying to find the right copula to be employed on their data. However, the proposed method as such may be of independent interest from theoretical point of view.
\end{abstract}

\thanks{All four authors acknowledge financial support from the Slovenian Research Agency (research core funding No. P1-0222).\\ \date}
\keywords{Copula; dependence concepts; imprecise copula; supremum and minimum of a set of copulas; asymmetry or non-exchangeability; measures of concordance}
\subjclass[2010]{Primary:     60E05; Secondary: 60E15, 62N05}

\maketitle

\section{ Introduction }

Copulas are mathematical objects that capture the dependence structure among random variables. Since they were introduced by A.\ Sklar in 1959 \cite{Skla} they have gained a lot of popularity and applications in several fields, e.g., in finance, insurance and reliability theory. Through them we model the dependence between random variables by building (bivariate) distributions with given marginal distributions. When deciding about which copulas to apply in real life scenarios the practitioners need to compare how certain statistical concepts behave on their data and on a class of copulas they intend to exploit.

An important family of such concepts form measures of concordance (cf.\ \cite{DuSe,Nels}) such as Kendall's tau and Spearman's rho or, slightly more generally, measures of association. On the other hand there is the notion of symmetry, also called exchangeability, or the lack of it, which also plays a crucial role in deciding about the choice of dependence rules practitioners want to utilize on their data. These notions have been studied extensively and increasingly. Let us refer to some studies on measures of association \cite{KoBuKoMoOm2, Krus,NeQuMoRoLaUbFl,NeUbFl} and on measures of nonexchangeability \cite{BBMNU14,DuKlSeUbFl,GeNe, KlMe,KoBuKoMoOm2,N}. Perhaps the most significant direction in development of recent studies is computing the local bounds of Fr\'{e}chet-Hoeffding type for families of copulas behaving equally or similarly with respect to a given measure. We need to point out the paper \cite{KoBuKoMoOm2} in this direction, where the authors study a family of copulas with given asymmetry at a given point (non-diagonal, of course) and then compute various measures of association for this family of copulas. This way they put each measure of association considered in a relation with asymmetry. Their point is that the studied family is narrowed down (although not quite determined) by its local bounds.

The main contribution of this paper is to give relations between a fixed measure of association, i.e., Blomqvist's beta, and all other important measures, i.e., Spearman's rho, Kendall's tau, Spearman's footrule, and Gini's coefficient gamma. Another relevant contribution is the development of a novel method by expanding the method of \cite{KoBuKoMoOm2} in order to find these relations. (Observe that, in particular, our method differs substantially from the methods developed in \cite{FrNe,GeNe2,ScPaTr} to study the relation between Kendall's tau and  Spearman's rho.) A third contribution that may be less important in view of applications but perhaps even more important from theoretical point of view is related to our approach as such and will be presented in more details in Section \ref{sec:max asym}.

The paper is organized as follows: Preliminaries on measures of concordance are presented in Section \ref{sec:prelim}, while our method is explained in Section \ref{sec:max asym}. The main results are presented in Sections \ref{sec:rho} (relations to Spearman's rho, Section \ref{sec:tau} (relations to Kendall's tau), Section \ref{sec:footrule} (relations to Spearman's footrule), and Section \ref{sec:gamma} (relations to Gini's gamma). At the end of the paper we give a conclusion presenting some ideas for further investigations.


\section{Preliminaries on measures of concordance}\label{sec:prelim}

A pair of random variables is \emph{concordant} if larger values of the first variable are associated with larger values of the second one, while smaller values of the first one are associated with smaller values of the second. The opposite notion is the notion of discordance. A pair of random variables is \emph{discordant} if larger values of the first variable are associated with smaller values of the second one, while smaller values of the first one are associated with larger values of the second. With this in mind, we denote by $\cQ$ (see \cite[\S 5.1]{Nels} or \cite[\S 2.4]{DuSe}) the difference of two probabilities $\cQ=P((X_1-X_2)(Y_1-Y_2)>0)-P((X_1-X_2)(Y_1-Y_2)<0)$ for a pair of random vectors $(X_1,X_2)$ and $(Y_1,Y_2)$. If the corresponding copulas are $C_1$ and $C_2$, then we have
\begin{equation}\label{concordance}
  \cQ=\cQ(C_1,C_2)= 4 \int_{\II^2} C_2(u,v)dC_1(u,v) -1.
\end{equation}
See \cite[Theorem 5.1.1]{Nels}. Function $\cQ$ is called the \emph{concordance function}. It was introduced by Kruskal \cite{Krus}. This function has a number of useful properties \cite[Corollary 5.1.2]{Nels}:
\begin{enumerate}
  \item It is symmetric in the two arguments.
  \item It is nondecreasing in each argument.
  \item It remains unchanged when both copulas are replaced by their survival copulas.
\end{enumerate}

We denote by $\CC$ the set of all bivariate copulas and by $\II$ the interval $[0,1]\subseteq\RR$. Recall that $C\leqslant D$ for $C,D\in\CC$ means that $C(u,v)\leqslant D(u,v)$ for all $(u,v)\in\II^2$. This introduces an order on $\CC$ which is called the \emph{pointwise order} in \cite[Definition 1.7.1]{DuSe}. Observe that the same order on copulas is denoted by $C\prec D$ and called \emph{concordance ordering} in \cite[Definition 2.8.1]{Nels}.
A mapping $\kappa:\CC\to [-1,1]$ is called a \emph{measure of concordance} if it satisfies the following properties (see \cite[Definition 2.4.7]{DuSe}):
\begin{enumerate}
\item $\kappa(C)=\kappa(C^t)$ for every $C\in\CC$.
\item $\kappa(C)\leqslant\kappa(D)$ when $C\leqslant D$.  \label{monotone}
\item $\kappa(\Pi)=0$, where $\Pi$ is the independence copula $\Pi(u,v)=uv$.
\item $\kappa(C^{\sigma_1})=\kappa(C^{\sigma_2})=-\kappa(C)$.
\item If a sequence of copulas $C_n$, $n\in\mathbb{N}$, converges uniformly to $C\in\CC$, then $\lim_{n\to\infty}\kappa(C_n)=\kappa(C)$.
\end{enumerate}
Here we have denoted by $C^t$ the transpose of $C$, i.e.. $C^t(u,v)= C(v,u)$, and by $C^{\sigma_1}$, respectively $C^{\sigma_2}$, the reflected copula of $C$, i.e., the copula obtained from it after sending $u\to 1-u$, respectively $v\to 1-v$.
We will refer to property (\ref{monotone}) above simply by saying that a measure of concordance under consideration is \emph{monotone}.\\

The five most commonly used measures of concordance of a copula $C$ are Kendall's tau, Spearman's rho, Spearman's footrule, Gini's gamma and Blomqvist's beta.

The first four of them may be defined in terms of the concordance function $\cQ$. Here we use the usual notation for the three standard copulas, i.e., the Fr\'{e}chet Hoeffding upper bound, respectively lower bound, $M(u,v)=\min\{u,v\}$, respectively $W=\max\{0,u+v-1\}$, and the product copula $\Pi(u,v)=uv$.
The \emph{Kendall's tau} of $C$ is defined by
\begin{equation}\label{tau}
\tau(C)=\cQ(C,C),
\end{equation}
\emph{Spearman's rho} by
\begin{equation}\label{rho}
\rho(C)=3\cQ(C,\Pi),
\end{equation}
\emph{Gini's gamma} by
\begin{equation}\label{gamma}
\gamma(C)=\cQ(C,M)+\cQ(C,W),
\end{equation}
\emph{Spearman's footrule} by
\begin{equation}\label{phi}
\phi(C)=
\frac12\left(3\cQ(C,M)-1\right).
\end{equation}
On the other hand, \emph{Blomqvist's beta} is defined by
\begin{equation}\label{beta}
\beta(C)=
4C\left(\frac12,\frac12\right)-1.
\end{equation}
See \cite[{\S}2.4]{DuSe} and \cite[Ch. 5]{Nels} for further details.

\section{  An important imprecise copula }\label{sec:max asym}

In this Section we will present our method in details.
Besides the Fr\'{e}chet-Hoeffding upper and lower bound, which are global bounds for the ordered set of copulas one often studies local bounds of certain subsets. Perhaps among the first known examples of the kind is given in Theorem 3.2.3 of Nelsen's book \cite{Nels} (cf.\ also \cite[Theorem 1]{NeQuMoRoLaUbFl}, where the bounds of the set of copulas $C\in\CC$ with $C(a,b)=\theta$ for fixed $a,b\in\II$ and $\theta\in[W(a,b),M(a,b)]$ are given). In general, if $\mathcal{C}_0$ is a set of copulas, we let
\begin{equation}\label{eq:inf:sup}
  \underline{C}=\inf\mathcal{C}_0\quad\quad\overline{C} =\sup\mathcal{C}_0.
\end{equation}
In \cite{NeQuMoRoLaUbFl} the authors study the bounds for the set of copulas whose Kendall's tau equals a given number $t\in [-1,1]$ and for the set of copulas whose Spearman's rho equals a given number $t\in [-1,1]$. In both cases the bounds are copulas that do not belong to the set. Similar bounds for the set of copulas having a fixed value of Blomqvist's beta were found in \cite{NeUbFl}. In \cite{BBMNU14} the authors present the local bounds for the set of copulas having a fixed value of the degree of non-exchangeability.
The authors of \cite{KoBuKoMoOm2} expand this idea further and develop a method that serves as a raw model to our approach as well, so let us explain its specifics. (We will follow the terminology of \cite{KoBuKoMoOm2} by calling a family of copulas such as $\mathcal{C}_0$ above determined by the local bounds \eqref{eq:inf:sup} and some additional conditions an \emph{imprecise copula}.)

The notion \emph{maximal asymmetry function} was introduced in \cite[\S 2]{KoBuKoMoOm1} following the ideas of \cite{KlMe}, its value at a fixed point $(u, v)\in\II^2$ was computed as
\[
    d^*_\FF (u,v) = \sup_{C\in\FF} \{|C(u,v)-C(v,u)|\},
\]
where $\FF\subseteq\CC$ is an arbitrary family of copulas. If $\FF=\CC$, this supremum is attained since $\CC$ is a compact set by \cite[Theorem 1.7.7]{DuSe}. Klement and Mesiar \cite{KlMe} and Nelsen \cite{N} showed that
\begin{equation}\label{eq:kle mes}
  d_\CC^*(u,v)=\min\{u,v,1-u,1-v,|v-u|\}.
\end{equation}
In \cite{KoBuKoMoOm2} an imprecise copula was introduced as follows: choose $(a,b)\in\II^2$ and a $c\in\II$ such that $0\leqslant c\leqslant d_\CC^*(a,b)$. Define $\CC_0$ to be the set of all $C$ such that
\begin{equation}\label{eq:asym_point}
  C(a,b)-C(b,a) = c.
\end{equation}
Note that this set is nonempty since the set $\CC$ is convex by
\cite[Theorem 1.4.5]{DuSe}. The local bounds $\underline{C}$ and $\overline{C}$ of this set were computed in \cite[Theorem 1]{KoBuKoMoOm2}

\begin{equation}\label{eq:asym cop left}
    \underline{C}^{a,b,c}(u,v)= \max\{W(u,v),\min\{d_1,u-a+d_1,v-b+d_1,u+v-a-b+d_1\}\},
\end{equation}
and
\begin{equation}\label{eq:asym cop right}
    \overline{C}^{b,a,c}(u,v)=\min\{M(u,v),\max\{d_2,u-b+d_2,v-a+d_2,u+v-a-b+d_2\}\},
\end{equation}
where
\begin{equation}\label{d_1}d_1=W(a,b)+c,\end{equation}
and
\begin{equation}\label{d_2}d_2=M(a,b)-c.\end{equation}
Observe that $c$ is small enough so that everywhere close to the boundary of the square $\II^2$ copula $W$ prevails in the definition of $\underline{C}^{a,b,c}$. The proof for $\overline{C}^{b,a,c}$ goes similarly. Note that $\underline{C}^{a,b,c}$ and $\overline{C}^{b,a,c}$ are shuffles of $M$, compare \cite[{\S}3.2.3]{Nels} and  \cite[\S3.6]{DuSe} (cf.\ also \cite{N}), so they are automatically copulas. It is also clear that $\underline{C}^{a,b,c}$ satisfies Condition \eqref{eq:asym_point}, since $\underline{C}^{a,b,c}(b,a)=W(b,a)$ and $\underline{C}^{a,b,c}(a,b)=d_1=W(a,b)+c$. The fact that $\overline{C}^{b,a,c}$ satisfies this condition goes in a similar way using the definition of $d_2$.

\textbf{Remark.} Observe that these bounds of the imprecise copula $\mathcal{C}_0$ are shuffles of $M$ according to \cite[{\S}3.2.3]{Nels} and all local bounds of imprecise copulas that we found in the literature so far are shuffles of $M$.

To compute the values of various measures of concordance we need the values of $\cQ$ introduced in Section \ref{sec:prelim} for various copulas such as $W$, $\Pi$, $M$, and $\underline{C}^{a,b,c}$, respectively $\overline{C}^{b,a,c}$. Recall that $d_1$ and $d_2$ are given by (\ref{d_1}) and \eqref{d_2}.

The following proposition is proved in \cite[Proposition 3\&4] {KoBuKoMoOm2}
. It was also pointed out there that these results are symmetric with respect to the main diagonal and to the counter-diagonal.

\begin{proposition} \label{prop1}
Let $(a,b)\in\II^2$ and $0\leqslant c\leqslant d_\CC^*(a,b)$. For copulas $\underline{C}^{a,b,c}$ and $\overline{C}^{b,a,c}$ it holds:
\begin{enumerate}[(a)]
\item $\cQ(W, \underline{C}^{a,b,c}) = 4d_1(1-a-b+d_1) -1,$
\item $\cQ(\Pi, \underline{C}^{a,b,c}) = 2d_1(1-a-b+d_1)(1-a-b+2d_1)-\dfrac13,$
\item $\cQ(\underline{C}^{a,b,c}, \underline{C}^{a,b,c}) = 4d_1(1-a-b+d_1) -1.$
\item $\cQ(W, \overline{C}^{b,a,c}) =
        (a-1)^2+(b-1)^2+2d_2(a+b-d_2)-1$  if $d_2 \leqslant \min \{ 1-a,1-b,2a+b-1,a+2b-1\}$
\item $\cQ(\Pi, \overline{C}^{b,a,c}) = \dfrac13-2(a+b-2d_2)(a-d_2)(b-d_2),$
\item $\cQ(\overline{C}^{b,a,c}, \overline{C}^{b,a,c}) = 1 - 4(a-d_2)(b-d_2).$
\item $\cQ(M, \overline{C}^{b,a,c}) = 1 - 4(a-d_2)(b-d_2),$
\end{enumerate}
\end{proposition}

As we have already observed, copulas $\underline{C}^{a,b,c}$ and $\overline{C}^{b,a,c}$ are shuffles of $M$, so that \eqref{eq:asym cop left} and \eqref{eq:asym cop right} can be rewritten as
\begin{equation}\label{eq:C shuffle}
  \begin{split}
     \underline{C}^{a,b,c} & =M(4,\{[0,a-d_1],[a-d_1,a],[a, 1-b+d_1], [1-b+d_1,1],(4,2,3,1),-1\})\\
     \overline{C}^{a,b,c}  & =M(4,\{[0,d_2],[d_2,b],[b,a+b-d_2], [a+b-d_2,1],(1,3,2,4),1\}),
  \end{split}
\end{equation}
where the last parameter in the above expression for the shuffle of $M$ is a function $f:\{1,2,\ldots,n\}\to\{-1,1\}$ which is in the first line of Equation \eqref{eq:C shuffle} identically equal to $-1$ and in the second one identically equal to 1.

We recall the imprecise copula of \cite{NeUbFl}
\begin{equation}\label{eq:beta}
  \BB_t:=\{C\in\CC\,|\,\beta(C)=t\}=\left\{C\in\CC\,\left|\, C\left(\dfrac12,\dfrac12\right)=\dfrac{1+t}{2}\right.\right\}\quad \mbox{for}\quad t\in [-1,1]
\end{equation}
and relate it to the imprecise copula defined by Equation \eqref{eq:asym_point}, actually we will relate its local bounds $\underline{B}_t=\inf\BB_t$ and $\overline{B}_t=\sup\BB_t$ to the bounds given by \eqref{eq:asym cop left} and \eqref{eq:asym cop right}.

\textbf{Remark.} Here comes the main point of our method. Although the imprecise copula $\mathcal{B}_t$ defined by \eqref{eq:beta} is a completely different family of copulas as the imprecise copula $\mathcal{C}_0$ defined in \cite{KoBuKoMoOm2} and explained above, its local bounds can be computed, somewhat surprisingly, as a special case of the local bounds of $\mathcal{C}_0$. This fact will be proven now and will serve as the basis of our investigation.

\begin{lemma}\label{lem1} The local bounds of the imprecise copula defined by Equation \eqref{eq:beta} can be expressed as special cases of copulas defined by Equations \eqref{eq:asym cop left} and \eqref{eq:asym cop right}:
  \begin{enumerate}[(a)]
    \item $\underline{B}_t=\underline{C}^{\frac12,\frac12, \frac{1+t}{4}}$,
    \item $\overline{B}_t=\overline{C}^{\frac12,\frac12, \frac{1-t}{4}}$.
  \end{enumerate}
\end{lemma}

\begin{proof}
  Using \eqref{eq:C shuffle} we get
  \begin{equation*}
  \begin{split}
     \underline{C}^{\frac12,\frac12, \frac{1+t}{4}} & =M\left(4,\left\{\left[0,\frac{1-t}{4}\right],\left[\frac{1-t}{4},\frac12 \right], \left[\frac12, \frac{3+t}{4}\right],\left[ \frac{3+t}{4},1\right], (4,2,3,1),-1\right\}\right)\\
     \overline{C}^{\frac12,\frac12, \frac{1-t}{4}}  & =M\left(4,\left\{\left[0,\frac{1-t}{4}\right],\left[\frac{1-t}{4},\frac12 \right], \left[\frac12, \frac{3+t}{4}\right],\left[ \frac{3+t}{4},1\right],(1,3,2,4),1\right\}\right),
  \end{split}
\end{equation*}
  Following \cite[Theorem 1]{NeUbFl} we have
  \[
    \underline{B}_t(u,v)=\max\left\{0,u+v-1,\frac{1+t}{4}- \left(\dfrac12-u\right)^+-\left(\dfrac12-v\right)^+\right\},\ \ \underline{B}_t\left(\frac12,\frac12\right)= \frac{1+t}{4},
  \]
  which amounts to the same as above at the points given after a short computation. In a similar way we conclude
  \[
    \overline{B}_t(u,v)=\min\left\{u,v,\frac{1+t}{4}+ \left(u-\dfrac12\right)^++\left(v-\dfrac12\right)^+\right\},\ \ \overline{B}_t\left(\frac12,\frac12\right)= \frac{1+t}{4},
  \]
and the same can be computed from the above at the points given.
\end{proof}

This lemma will enable us to transmit the results of computations of \cite{KoBuKoMoOm2} presented in Proposition \ref{prop1} in determining the imprecise copula given by \eqref{eq:beta}.

\section{ Blomqvist's beta vs.\ Spearman's rho}\label{sec:rho}

In this section we find all possible pairs $(\beta(C),\rho(C))$ for a copula $C$. First we compute the values of $\rho$ at the bounds of the imprecise copula $\mathcal{B}_t$.

\begin{theorem}\label{thm:rho} Given any $t\in[-1,1]$ the value of $\rho$ is bounded by:
  \begin{enumerate}[(a)]
    \item $\rho(\underline{B}_t)=\dfrac{3}{16}(1+t)^3-1$, and
    \item $\rho(\overline{B}_t)=1-\dfrac{3}{16}(1-t)^3$.
  \end{enumerate}
\end{theorem}

\begin{proof}
  Using first Equation \eqref{rho}, then Lemma \ref{lem1}\emph{(a)}, and finally Proposition \ref{prop1}\emph{(b)}, we show that
  \[
  \begin{split}
     \rho(\underline{B}_t) & =3\cQ(\underline{B}_t,\Pi) =3\cQ(\underline{C}^{\frac12,\frac12, \frac{1+t}{4}},\Pi) \\
       &=6d_1(1-a-b+d_1)(1-a-b+2d_1)-1=6c\cdot c\cdot (2c)-1.
  \end{split}
  \]
  To get item \emph{(a)} observe that $c=\dfrac{1+t}{4}$. Next, we follow a similar pattern in proving item \emph{(b)}, Equation \eqref{rho}, then Lemma \ref{lem1}\emph{(b)}, and finally Proposition \ref{prop1}\emph{(e)} in order to find:
  \[
  \begin{split}
     \rho(\overline{B}_t) & =3\cQ(\overline{B}_t,\Pi) =3\cQ(\overline{C}^{\frac12,\frac12, \frac{1-t}{4}},\Pi) \\
       &=1-6(a+b-2d_2)(a-d_2)(b-d_2) =1- 6(1-2d_2)\left(\dfrac12-d_2 \right)^2\\
       &=1-12c^3.
  \end{split}
  \]
  Finally, observe that $c=\dfrac{1-t}{4}$.
\end{proof}

Figure 1 exhibits the set of all possible  pairs $(\beta(C),\rho(C))$ for a copula $C$. The expressions for the bounds of the shaded regions are given in Theorem \ref{thm:rho} and the following corollary.

\begin{figure}[h]
            \includegraphics[width=5cm]{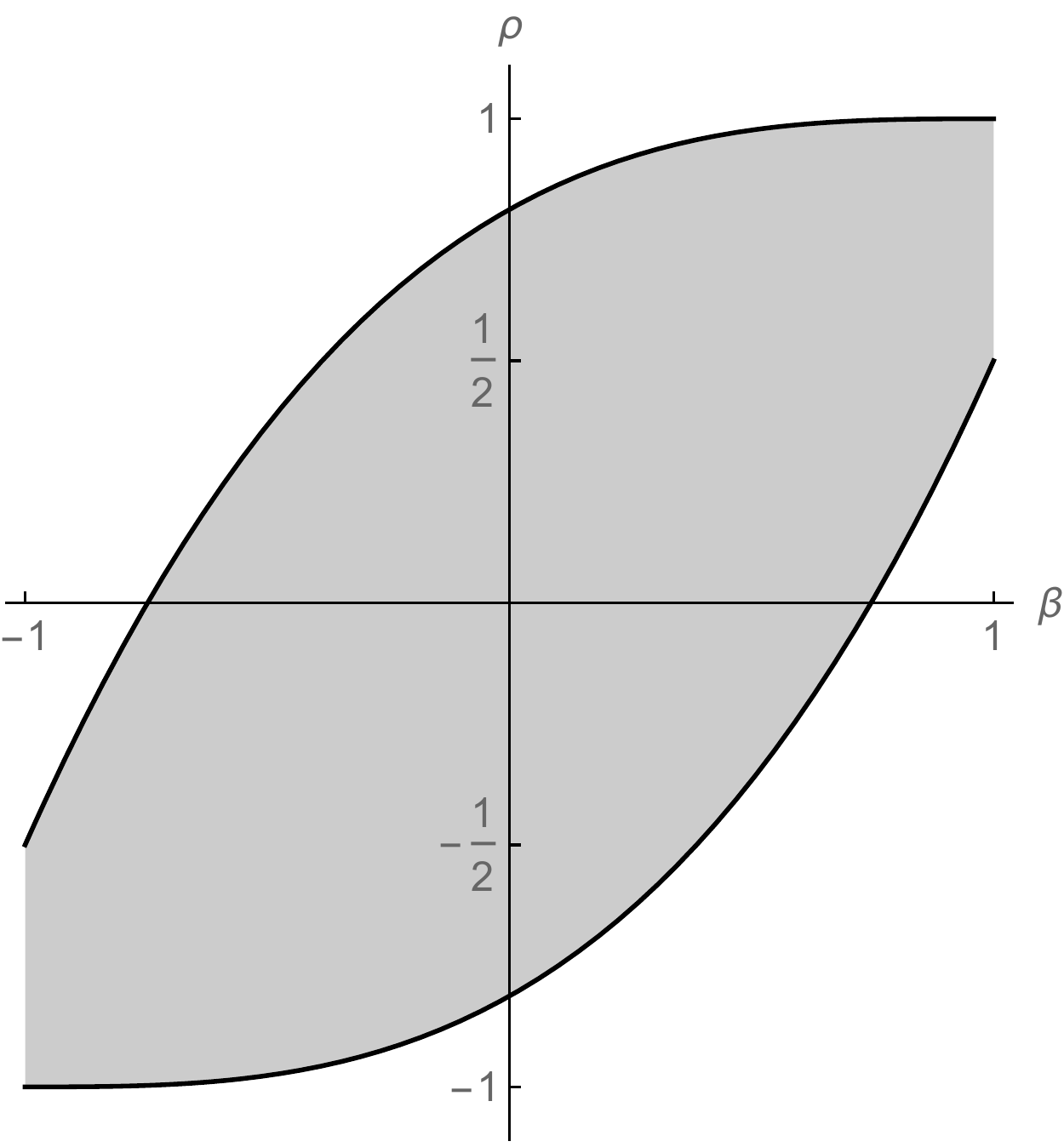} \hfil \includegraphics[width=5cm]{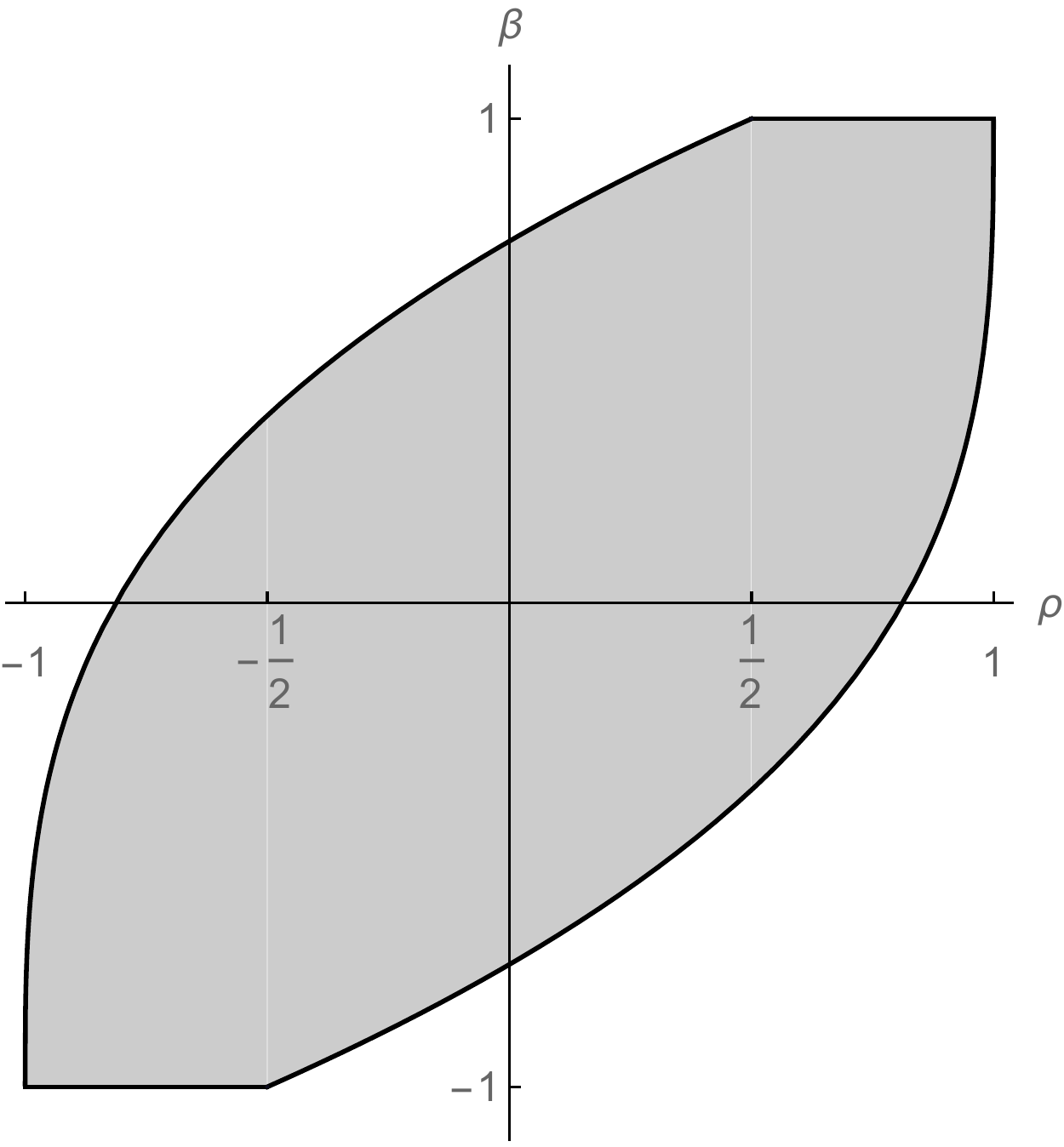}
            \caption{ Blomqvist's beta vs.\ Spearman's rho }\label{fig1}
\end{figure}

\begin{corollary}
  If $\rho(C)=\rho$ for some $C\in\CC$ and $\rho\in[-1,1]$, then
  \[
    \left.\begin{matrix}
      -1, & \mbox{if } \rho\leqslant-\dfrac12 \\
      1-2\sqrt[3]{\dfrac{2(1-\rho)}{3}}, & \mbox{otherwise}
    \end{matrix}\right\}\leqslant \beta(C)\leqslant
    \begin{cases}       1, & \mbox{if } \rho\geqslant\dfrac12 \\       -1+2\sqrt[3]{\dfrac{2(1+\rho)}{3}}, & \mbox{otherwise}.     \end{cases}
  \]
\end{corollary}

\section{ Blomqvist's beta vs.\ Kendall's tau}\label{sec:tau}

In this section we study the relation between Blomqvist's beta and Kendall's tau on the set of all copulas. First we determine the value of $\tau$ at the bounds of the set $\mathcal{B}_t$ for any possible $t$.


\begin{theorem}\label{thm:tau} Given any $t\in[-1,1]$ it holds that
  \begin{enumerate}[(a)]
    \item $\tau(\underline{B}_t)=\dfrac{(1+t)^2}{4}-1$
    \item $\tau(\overline{B}_t)=1-\dfrac{(1-t)^2}{4}$.
  \end{enumerate}
\end{theorem}

\begin{proof}
  We compute, using first Equation \eqref{tau}, then Lemma \ref{lem1}\emph{(a)} and finally Proposition \ref{prop1}\emph{(c)}, that
  \[
  \begin{split}
     \tau(\underline{B}_t) & =\cQ(\underline{B}_t, \underline{B}_t) =\cQ(\underline{C}^{\frac12,\frac12, \frac{1+t}{4}},\underline{C}^{\frac12,\frac12, \frac{1+t}{4}}) \\
       &=4d_1(1-a-b+d_1)-1=4c^2-1=\dfrac{(1+t)^2}{4}-1.
  \end{split}
  \]
  To get item \emph{(a)} observe that $c=\dfrac{1+t}{4}$. Next, we follow a similar pattern in proving item \emph{(b)}: Equation \eqref{tau}, then Lemma \ref{lem1}\emph{(b)}, and finally Proposition \ref{prop1}\emph{(f)}, yield
  \[
  \begin{split}
     \rho(\overline{B}_t) & =3\cQ(\overline{B}_t,\Pi) =3\cQ(\overline{C}^{\frac12,\frac12, \frac{1-t}{4}},\Pi) \\
       &=1-4(a-d_2)(b-d_2) =1- 4c^2,
  \end{split}
  \]
  and observe at the end that $c=\dfrac{1-t}{4}$.
\end{proof}

Figure 2 depicts the set of all possible  pairs $(\beta(C),\tau(C))$ for a copula $C$. The expressions for the bounds of the shaded regions are given in Theorem \ref{thm:tau} and Corollary \ref{cor:tau}.

\begin{figure}[h]
            \includegraphics[width=5cm]{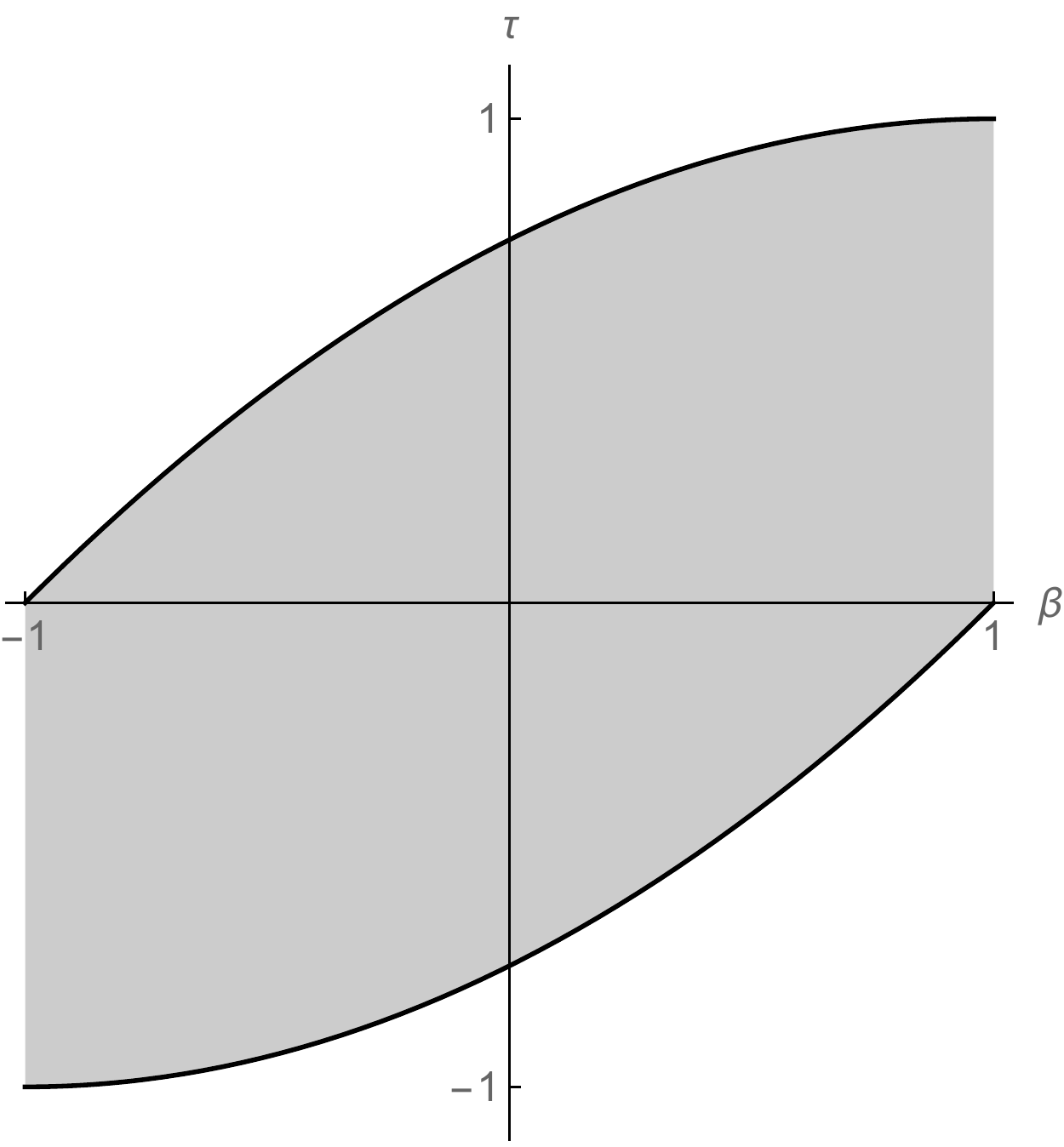} \hfil \includegraphics[width=5cm]{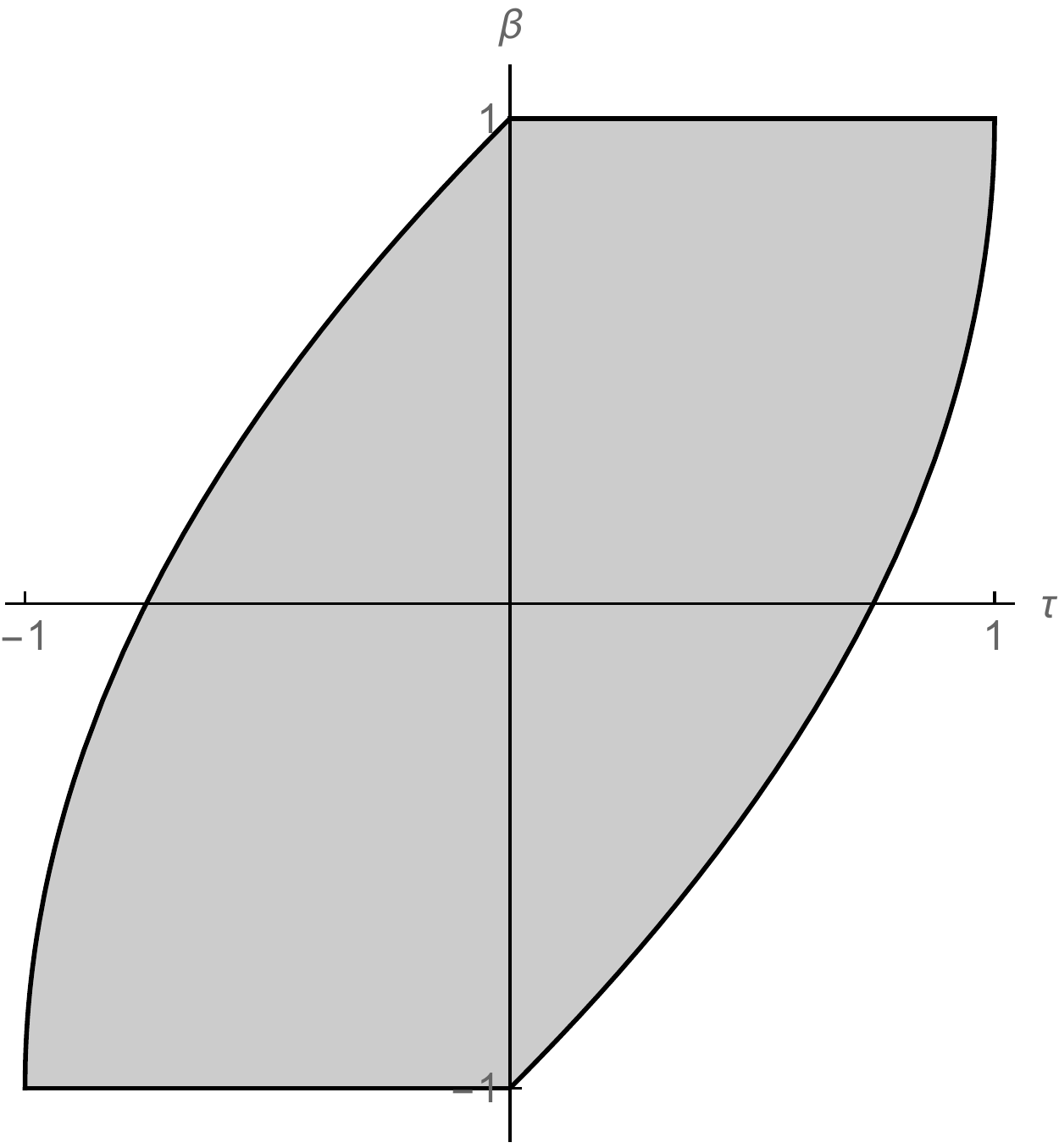}
            \caption{ Blomqvist's beta vs.\ Kendall's tau }\label{fig2}
\end{figure}

\begin{corollary}\label{cor:tau}
  If $\tau(C)=\tau$ for some $C\in\CC$ and $\tau\in[-1,1]$, then
  \[
    \left.\begin{matrix}
      -1, & \mbox{if } \tau\leqslant0 \\
      1-2\sqrt{1-\tau}, & \mbox{otherwise}
    \end{matrix}\right\}\leqslant \beta(C)\leqslant
    \begin{cases}       1, & \mbox{if } \tau\geqslant0\\       -1+2\sqrt{1+\tau}, & \mbox{otherwise}.     \end{cases}
  \]
\end{corollary}

\section{Blomqvist's beta vs.\ Spearman's footrule}\label{sec:footrule}

In order to compute Spearman's footrule we will first insert copula $C=\underline{C}^{a,b,c}$ into Equation \eqref{phi}. So, we start by computing
\[
    \begin{split}
       \cQ(\underline{C}^{a,b,c},M) & = -1+ 4\int_{0}^{a-d_1} M(u,4u)\,du+4\int_{a-d_1}^{a}M(u,a+b-d_1-u)\,du  \\
         &+ 4\int_{a}^{1-b+d_1}M(u,1+d_1-u)\,du+ 4\int_{1-b+d_1}^{1}M(u,1-u)\,du.
    \end{split}
\]
Next we apply Lemma \ref{lem1}\emph{(a)} to get
\[
    \begin{split}
       \cQ(\underline{B}_t,M) &=\cQ(\underline{C}^{\frac12,\frac12,\frac{1+t}{4}},M) = -1+ 4\int_{0}^{\frac12-\frac{1+t}{4}} u\,du+ 4\int_{\frac12-\frac{1+t}{4}}^{\frac12} M\left(u,1-\frac{1+t}{4}-u\right)\,du  \\
         &+ 4\int_{\frac12}^{\frac12+\frac{1+t}{4}} M\left(u,1+\frac{1+t}{4}-u\right)\,du+ 4\int_{\frac12+\frac{1+t}{4}}^{1}(1-u)\,du   \\
         &= -1+ 4\int_{0}^{\frac{3-t}{8}} u\,du+ 4\int_{\frac{3-t}{8}}^{\frac12} \left(\frac{3-t}{4}-u\right)\,du  + 4\int_{\frac12}^{\frac{5+t}{8}} u\,du\\
         &+ 4\int_{\frac{5+t}{8}}^{\frac{5+t}{4}} \left(\frac{3+t}{4} -u\right)\,du+ 4\int_{\frac{5+t}{8}}^{1}(1-u)\,du.
    \end{split}
\]
On the second step of these computations we needed a careful examination of which one of the two functions in the arguments of copula $M$ is smaller resulting in a rearrangement of the intervals of integration. A straightforward computation now brings us to
\begin{equation}\label{eq:integral}
  \cQ(\underline{B}_t,M)=\frac{(1+t)^2}{8}\quad\mbox{and}\quad \phi(\underline{B}_t)=\frac{3(1+t)^2}{16}-\dfrac12.
\end{equation}

\begin{theorem}\label{thm:phi} Given any $t\in[-1,1]$ we have
  \begin{enumerate}[(a)]
    \item $\phi(\underline{B}_t)=\dfrac{3(1+t)^2}{16}-\dfrac12$
    \item $\phi(\overline{B}_t)=1-\dfrac{3(1-t)^2}{8}$.
  \end{enumerate}
\end{theorem}

\begin{proof}
  Item \emph{(a)} was proven above. In the proof of item \emph{(b)}, we use Equation \eqref{phi}, then Lemma \ref{lem1}\emph{(b)} and finally Proposition \ref{prop1}\emph{(g)} to get:
  \[
  \begin{split}
     \phi(\overline{B}_t) & =\dfrac32\cQ(\overline{B}_t,M)-\frac12 =\dfrac32\cQ(\overline{C}^{\frac12,\frac12, \frac{1-t}{4}},M)-\frac12 \\
       &=\dfrac32(1- 4c^2)-\frac12=1-\dfrac38(1-t)^2.
  \end{split}
  \]
\end{proof}

In Figure 3 we display the set of all possible  pairs $(\beta(C), \phi(C))$ for a copula $C$. The expressions for the bounds of the shaded regions are given in Theorem \ref{thm:phi} and Corollary \ref{cor:phi}.

\begin{figure}[h]
            \includegraphics[width=5cm]{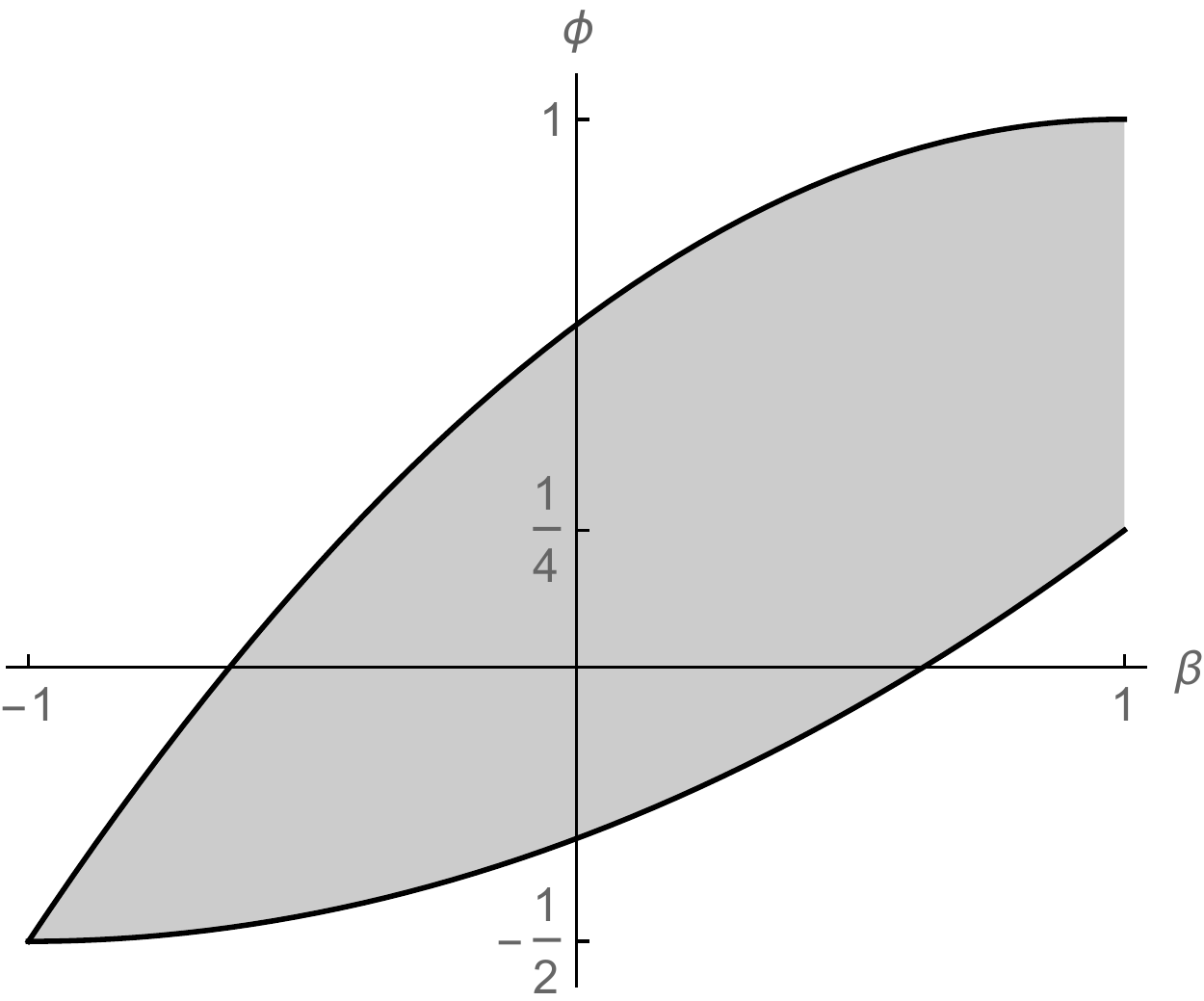} \hfil \includegraphics[width=5cm]{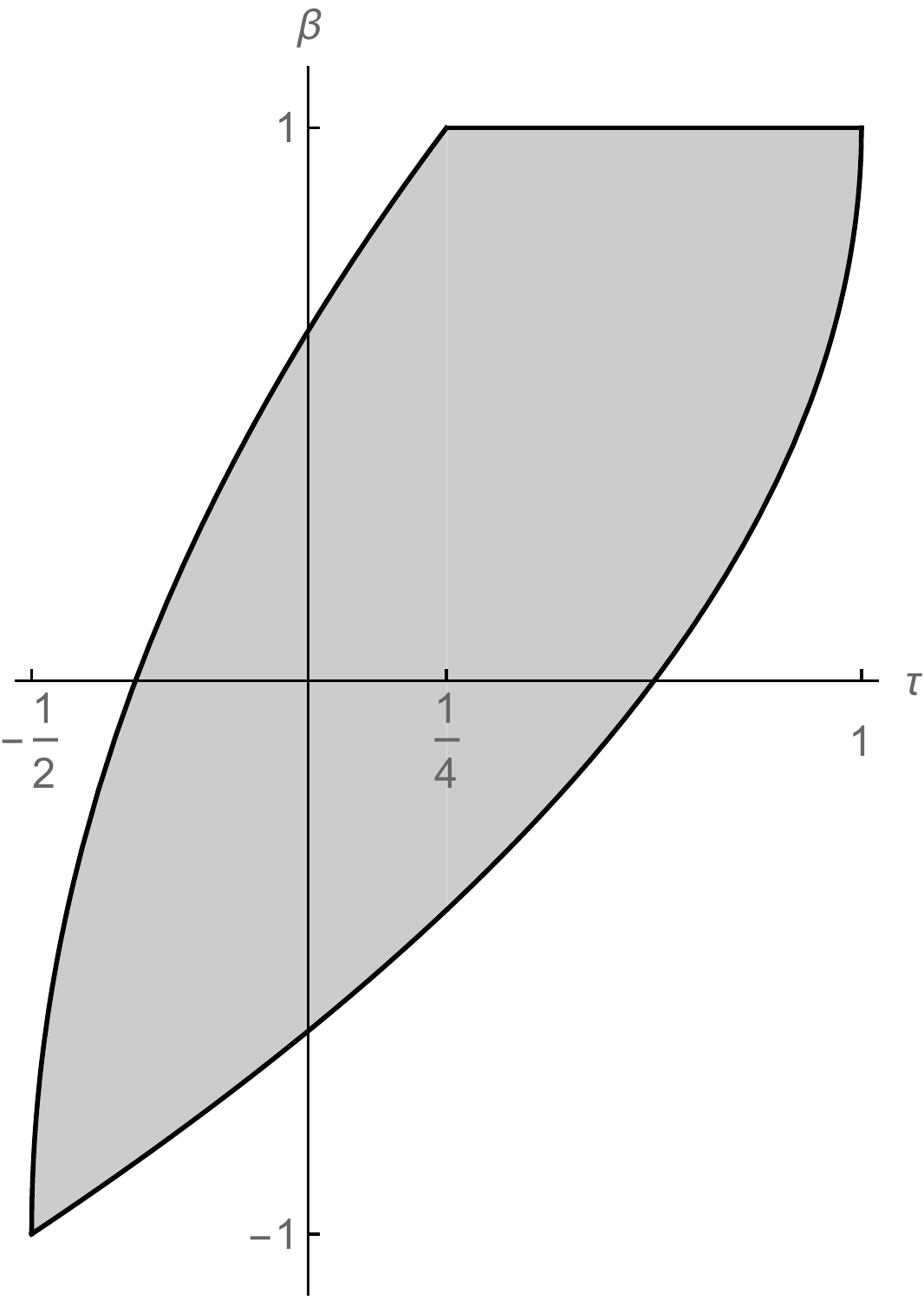}
            \caption{ Blomqvist's beta vs.\ Spearman's footrule }\label{fig3}
\end{figure}

\begin{corollary}\label{cor:phi}
  If $\phi(C)=\phi$ for some $C\in\CC$ and $\phi\in\left[-\dfrac12,1\right]$, then
  \[
    1-4\sqrt{\frac{1-\phi}{6}}\
    \leqslant \beta(C)\leqslant
    \begin{cases}       1, & \mbox{if } \dfrac14\leqslant\phi\leqslant1\\       -1+4\sqrt{\frac{1+2\phi}{6}}, & \mbox{otherwise}.     \end{cases}
  \]
\end{corollary}

\section{Blomqvist's beta vs.\ Gini's gamma}\label{sec:gamma}

In this section we consider the relationship between the Blomqvists's beta and Gini's gamma. We first compute Gini's gamma at the bounds of the set $\mathcal{B}_t$ for any $t\in [-1,1]$.

\begin{theorem}\label{thm:gamma} Given $t\in[-1,1]$ we have
  \begin{enumerate}[(a)]
    \item $\gamma(\underline{B}_t)=\dfrac{3(1+t)^2}{8}-1$
    \item $\gamma(\overline{B}_t)=1-\dfrac{3(1-t)^2}{8}$.
  \end{enumerate}
\end{theorem}

\begin{proof}
  We compute, using first 
  Lemma \ref{lem1}\emph{(a)} and then Proposition \ref{prop1}\emph{(a)}, that
  \[
  \begin{split}
     \cQ(\underline{B}_t,W)
     & =\cQ(\underline{C}^{\frac12,\frac12, \frac{1+t}{4}},W) =4c^2-1=\dfrac{(1+t)^2}{4}-1.
  \end{split}
  \]
  To get item \emph{(a)} use also Equation \eqref{gamma} and the left hand side of Equation \eqref{eq:integral}, so that
  \[
    \gamma(\underline{B}_t)=\cQ(\underline{B}_t,M)+ \cQ(\underline{B}_t,W)=\dfrac{(1+t)^2}{4}-1+\dfrac{(1+t)^2}{8}=-1+ \dfrac38(1+t)^2.
  \]
  In the proof of item \emph{(b)} we use Lemma \ref{lem1}\emph{(b)} and Proposition \ref{prop1}\emph{(e)}:
  \[
  \begin{split}
     \cQ(\overline{B}_t,M) =\cQ(\overline{C}^{\frac12,\frac12, \frac{1-t}{4}},M) =1-2c^2=1-\dfrac{(1-t)^2}{4}.
  \end{split}
  \]
  Also, by Proposition \ref{prop1}\emph{(d)} it follows that
  \[
    \cQ(\overline{B}_t,W)=-\frac{(1-t)^2}8{}.
  \]
  Now, use Equation \eqref{gamma} 
  to get
  \[
    \gamma(\overline{B}_t)=\cQ(\overline{B}_t,M)+ \cQ(\overline{B}_t,W)=1-\dfrac{(1-t)^2}{4}-\dfrac{(1-t)^2}{8}\ = 1-\dfrac38(1-t)^2
  \]
\end{proof}

Figure 4 exhibits the set of all possible pairs $(\beta(C), \gamma(C))$ for a copula $C$. The expressions for the bounds of the shaded regions are given in Theorem \ref{thm:gamma} and Corollary \ref{cor:gamma}.

\begin{figure}[h]
            \includegraphics[width=5cm]{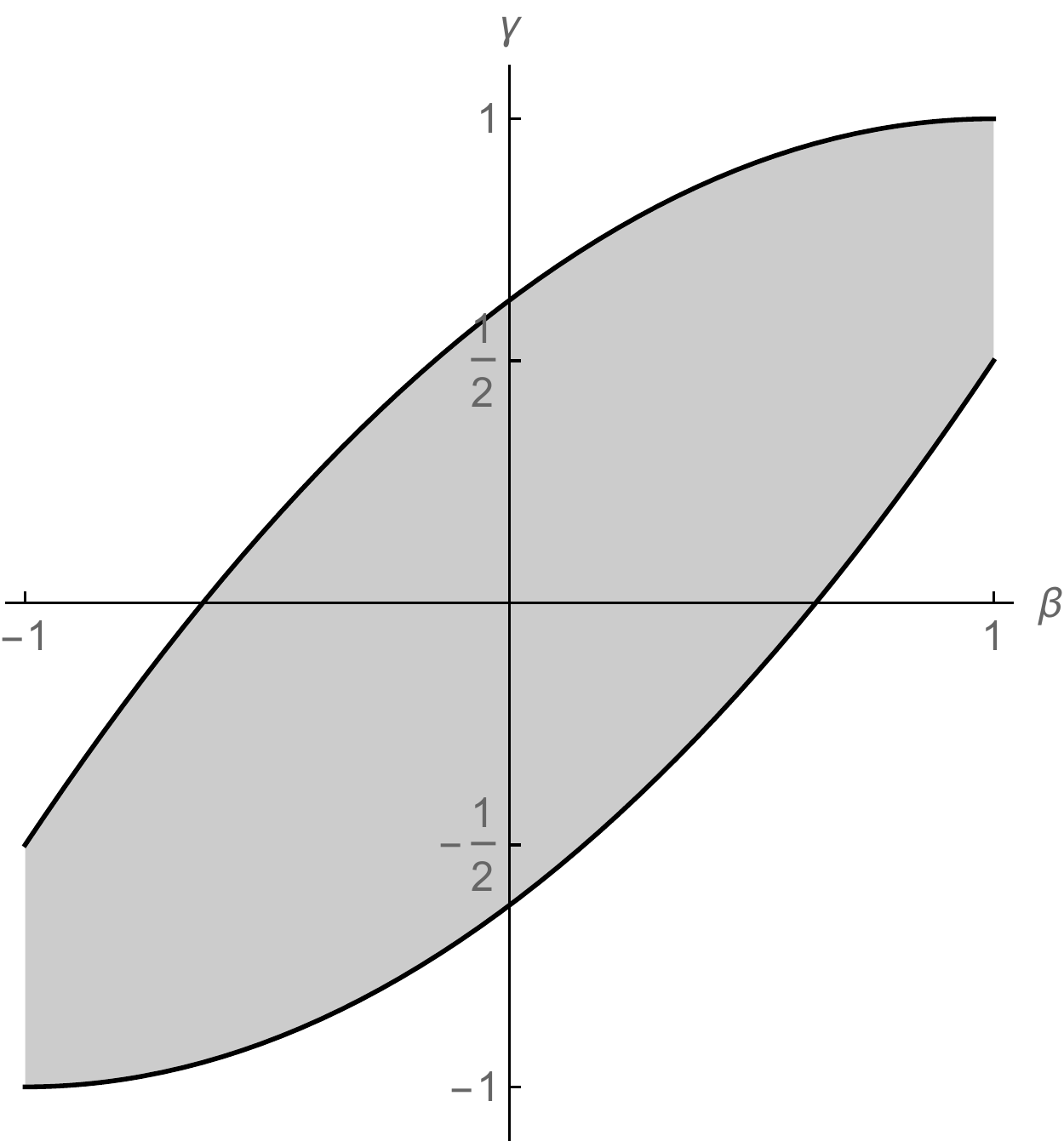} \hfil \includegraphics[width=5cm]{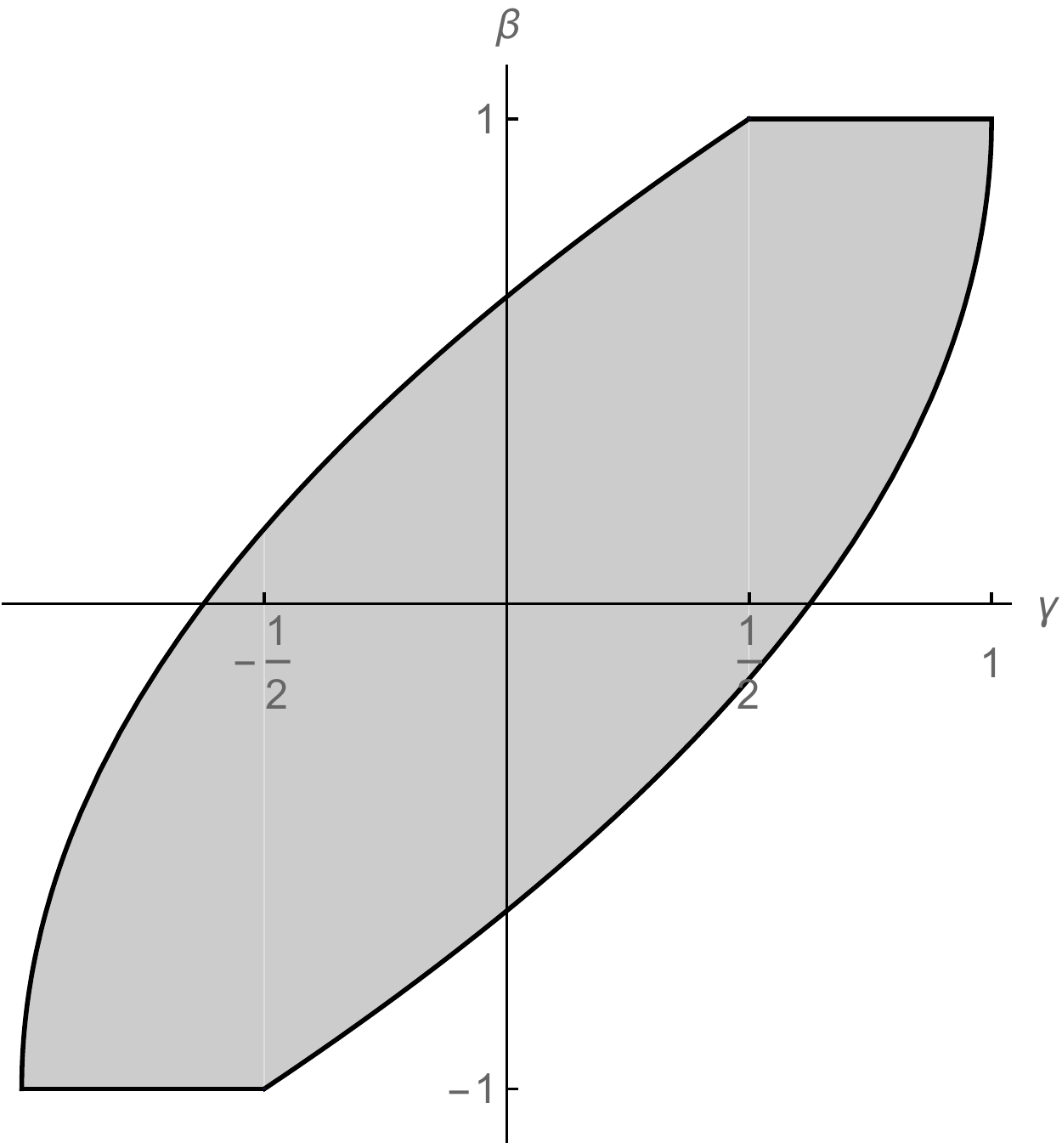}
            \caption{ Blomqvist's beta vs.\ Gini's gamma }\label{fig4}
\end{figure}

\begin{corollary}\label{cor:gamma}
  If $\gamma(C)=\gamma$ for some $C\in\CC$ and $\gamma\in[-1,1]$, then
  \[
    \left.\begin{matrix}
      -1, & \mbox{if } -1\leqslant\gamma\leqslant-\dfrac12 \\
      1-2\sqrt{\frac23(1-\gamma)}, & \mbox{otherwise}
    \end{matrix}\right\}
    \leqslant \beta(C)\leqslant
    \begin{cases}       1, & \mbox{if } \dfrac12\leqslant\gamma\leqslant1\\       -1+2\sqrt{\frac23(1+\gamma)}, & \mbox{otherwise}.     \end{cases}
  \]
\end{corollary}

\section{ Conclusion }\label{concl}

Our results can be explained in terms of imprecise copulas. Let us give some more details about this notion for the interested reader. Observe first that the pair $(\underline{C},\overline{C})$ defined by \eqref{eq:inf:sup} does not necessarily consist of two copulas even if $\mathcal{C}_0$ is made of copulas only. In general it is a pair of quasi-copulas that has certain properties and is called an \emph{imprecise copula} in \cite{MoMiPeVi}. Conversely, a question proposed there is whether any imprecise copula $(\underline{C}, \overline{C})$ satisfies \eqref{eq:inf:sup}. The question is answered in the negative in \cite{OmSt1} and equivalent conditions on the set of copulas in order to satisfy \eqref{eq:inf:sup} is given there. Imprecise copulas that do satisfy \eqref{eq:inf:sup} are said to be \emph{coherent} (cf.\ \cite{DoKoBuSkOm,OmSt2,OmSk}). The families of copulas used in our paper and called imprecise copulas may all be seen as coherent. However, unlike the one presented in Nelsen's book \cite[Theorem 2.3]{Nels}, it is not always clear whether they contain  all (!) copulas lying between the two local bounds.


The imprecise copula of \cite[Theorem 1]{BBMNU14}, say, does not necessarily contain all the copulas between the two bounds but only some of them, although it is coherent on the other hand as the authors prove in \cite[Theorem 2]{BBMNU14}. The kind of imprecise copulas in a narrower sense seems to have been introduced in \cite{OmSk}. There they emerge from the study of imprecise shock model copulas (cf. also \cite{DoKoBuSkOm}) such as imprecise Marshall's copulas and imprecise maxmin copulas. In all these cases including the ones studied in our paper the defining condition of a set in question is not fulfilled automatically by all the copulas between the local bounds as opposed to the case studied in \cite{Nels}.

The main point of this paper is built on the fact described in the remark just preceding Proposition \ref{prop1}, i.e., the fact that the imprecise copula $\mathcal{B}_t$ has the same local bounds as the imprecise copula $\mathcal{C}_0$ of Section \ref{sec:max asym}, actually a special case of it. This enables us to relate any measure of concordance, and more generally any monotone function $\kappa\colon\mathcal{C}\to[0,1]$ to Blomqvist's beta. We believe that this method can be expanded further to study relations between any pair of monotone functions on $\mathcal{C}$. However, this would require an even more sophisticated adjustment of our techniques.

As a final remark let us point out that the notion of imprecise copulas is borrowed from the imprecise approach into the copula theory. We know that most of the copula community or even more generally the probability community is reluctant to use imprecise notions since they stand firmly in the standard probability theory. Indeed, the imprecise community may be using finitely additive probability which results in probability distributions of random variables that are monotone functions only and not always cadlag. However, it is a side result of a recent paper \cite{MoMiPeVi} that every bivariate random vector (even if we start in a finitely additive probability space) can be expressed as a copula (i.e., the usual Sklar's copula) composed with possibly non-standard marginal distributions. So, whatever there is non-standard in a bivariate random vector, it moves to the marginal distributions, while copulas remain the same. Does this mean that copulas are a stronger probabilistic concept than the probability itself?


\begin{thebibliography}{}


\bibitem{BBMNU14} G. Beliakov, B. DeBaets, H. DeMeyer, R. B. Nelsen, M. \`Ubeda-Flores. {\it Best-possible bounds on the set of copulas with given degree of non-exchangeability}, J. Math. Anal. Appl. \textbf{417} (2014), 451--468.

\bibitem{DoKoBuSkOm} D. Dol\v{z}an, D.\ Kokol Bukov\v{s}ek, D.\ \v{S}kulj, M. Omladi\v{c}. {\it Some multivariate imprecise shock model copulas}, preprint.

\bibitem{DuKlSeUbFl} F. Durante, E. P. Klement, C. Sempi,  M. \'Ubeda-Flores. {\it Measures of non-exchangeability for bivariate random vectors.} Statist. Papers {\bf 51} (2010), no. 3, 687--699.

\bibitem{DuSe} F.~Durante, C.~Sempi. {\it Principles of Copula Theory}, CRC/Chapman \& Hall, Boca Raton (2015).

\bibitem{FrNe} G.\ A.\ Fredricks, R.\ B.\ Nelsen. \textsl{On the relationship between Spearman's rho and Kendall's tau for pairs of continuous random variables}, Journal of Statistical Planning and Inference \textbf{137} (2007), 2143--2150.

\bibitem{GeNe2}
C.\ Genest, J.\ Ne\v{s}lehov\'a. \textsl{Analytical proofs of classical inequalities between Spearman's $\rho$ and Kendall's $\tau$},  Journal of Statistical Planning and Inference \textbf{139} (2009), 3795--3798.

\bibitem{GeNe} C.\ Genest, J.\ Ne\v{s}lehov\'a. \textsl{Assessing and Modeling Asymmetry in Bivariate Continuous Data}. In: P.\ Jaworski, F.\ Durante, W.K.\ H\"ardle, (eds.), {C}opulae in {M}athematical and {Q}uantitative {F}inance, Lecture Notes in Statistics \textbf{213}, Springer Berlin Heidelberg,  (2013), 91--114.


\bibitem{KlMe} E. P. Klement, R. Mesiar. {\it How non-symmetric can a copula be?} Comment. Math. Univ. Carolin. {\bf 47} (2006), no. 1, 141--148.


\bibitem{KoBuKoMoOm1} D.\ Kokol Bukov\v{s}ek, T.\ Ko\v{s}ir, B.\ Moj\v{s}kerc, M. Omladi\v{c}. {\it Non-exchangeability of copulas arising from shock models}, J.\ of Comp.\ and Appl.\ Math., \textbf{358} (2019), 61--83. (See also {Erratum} published in J.\ of Comp.\ and Appl.\ Math., \textbf{365} (2020), https://doi.org/10.1016/j.cam.2019.112419 and corrected version at https://arxiv.org/abs/1808.09698v4).

\bibitem{KoBuKoMoOm2} D.\ Kokol Bukov\v{s}ek, T.\ Ko\v{s}ir, B.\ Moj\v{s}kerc, M. Omladi\v{c}. {\it Relation between non-exchangeability and measures of concordance of copulas}, available at https://arxiv.org/abs/1909.06648, preprint.

\bibitem{Krus}  W. H. Kruskal. {\it Ordinal measures of association}. J. Amer. Stat. Soc., \textbf{53} (1958), 814--861.


\bibitem{MoMiPeVi} I.\ Montes, E.\ Miranda, R.\ Pelessoni, P.\ Vicig. {\it Sklar's theorem in an imprecise setting}. Fuzzy Sets and Systems, \textbf{278} (2015), 48--66.


\bibitem{Nels} R.\ B.\ Nelsen. {\it An introduction to copulas}, 2nd edition, Springer-Verlag, New York (2006).

\bibitem{N} R. B. Nelsen. {\it Extremes of nonexchangeability}. Statist. Papers {\bf 48} (2007), no. 2, 329--336.


\bibitem{NeQuMoRoLaUbFl} R. B. Nelsen, J. J. Quesada-Molina, J. A. Rodr\'{\i}guez-Lallena, M. \'{U}beda-Flores. {\it Bounds on bivariate distribution functions with given margins and measures of association}. Commun. Statist. Theory Meth., \textbf{30}(6) (2001), 1155--1162.


\bibitem{NeUbFl} R. B. Nelsen, M. \'{U}beda-Flores. {\it A comparison of bounds on sets of joint distribution functions derived from various measures of association.} Commun. Statist. Theory Meth., \textbf{33}(10) (2004), 2299--2305.











\bibitem{OmSt1} M. Omladi\v{c}, N. Stopar. {\it Final solution to the problem of relating a true copula to an imprecise copula}. Fuzzy sets and systems, 2019. [DOI information: 10.1016/j.fss.2019.07.002]

\bibitem{OmSt2} M. Omladi\v{c}, N. Stopar. {\it A full scale Sklar's theorem in the imprecise setting}, preprint.

\bibitem{OmSk} M.  Omladi\v{c}, D. \v{S}kulj. {\it Constructing copulas from shock models with imprecise distributions}, available at https://arxiv.org/abs/1812.07850, preprint.





\bibitem{ScPaTr} M. Schreyer, R. Paulin, W. Trutschnig. \emph{On the exact region determined by Kendall's $\tau$ and Spearman's $\rho$}, J. Royal Stat. Soc., Series B (Stat. Methodology), \textbf{79}, (2017), 613--633.


\bibitem{Skla} A.\ Sklar. {\it Fonctions de r\'{e}partition \`{a} $n$ dimensions et leurs marges}. Publ.\ Inst.\ Stat.\ Univ.\ Paris \textbf{8} (1959) 229--231.



\end{thebibliography}
\end{document}